\documentclass[11pt]{article} 
\usepackage{latexsym}
\usepackage{amsfonts, manfnt}
\usepackage[margin=30pt,font=small,labelfont=bf, labelsep=endash]{caption}
\usepackage{amsthm}
\usepackage{amsmath}
\usepackage{amssymb}
\usepackage{setspace}
\usepackage{eucal}
\usepackage{graphicx}
\usepackage[all, poly, knot]{xy}
\usepackage{verse, play}
\usepackage{graphicx, graphics}
\usepackage{amscd,amssymb,hyperref, latexsym}
\usepackage[all]{xy}
\usepackage{amssymb, amsmath}
\usepackage{epstopdf}
\input xy
\xyoption{all}
\xyoption{knot}
\xyoption{arc}

\newtheorem{theorem}{Theorem}
\newtheorem{lemma}[theorem]{Lemma}
\theoremstyle{definition}
\newtheorem{definition}[theorem]{Definition}

\theoremstyle{remark}

\numberwithin{equation}{section}
\theoremstyle{plain}

\newtheorem{corollary}[theorem]{Corollary}

\newtheorem{proposition}[theorem]{Proposition}

\title{The Link Smoothing Game}
\author{Allison Henrich and Inga Johnson}
\begin{document}
\maketitle

\abstract{We introduce a topological combinatorial game called the Link Smoothing Game. The game is played on the shadow of a link diagram and legal moves consist of smoothing precrossings. One player's goal is to keep the diagram connected while the other player's goal is to disconnect the shadow. We make significant progress towards a complete classification of link shadows into outcome classes by capitalizing on the relationship between link shadows and the planar graphs associated to their checkerboard colorings.}

\section{Introduction.}
Recently, several topological combinatorial games related to knots were introduced in~\cite{CMJ}. Inspired by the discovery of these games, we introduce a new game called the \textbf{Link Smoothing Game}.

Suppose $D$ is a connected diagram of a link shadow, that is, a connected link diagram where under- and over-strand information is unspecified at the crossings.  Two players take turns selecting a precrossing (i.e. an undetermined crossing) of the shadow and replacing it with either a horizontal or a vertical smoothing, as in Figure~\ref{smooth}.  

\begin{figure}[htbp] 
\begin{center} \includegraphics[scale=.25]{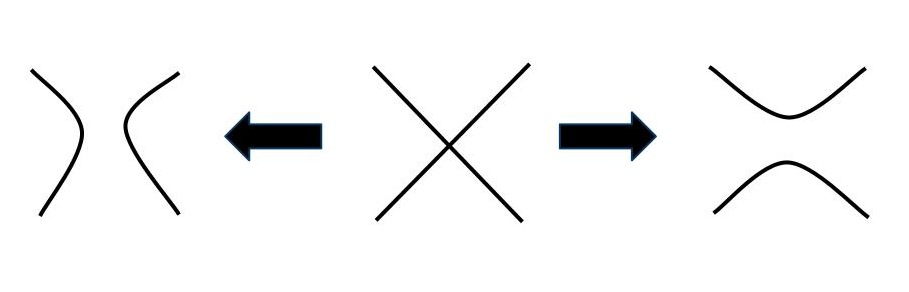}\end{center}
\caption{{\bf Smoothing a precrossing.}}
\label{smooth}
\end{figure}

The goal of one player is to keep the diagram connected while the goal of the other player is to disconnect the diagram.  The player with the goal of a keeping the diagram connected will be called Knot or $K$, while the other player will be called Link or $L$.  An example game is played in Figure~\ref{gameplay}.  

We note that the Link Smoothing Game is not a classical combinatorial game in which the winner is determined by who made the last move.  Rather it is a combinatorial game of a topological nature, as the goal is to change the game board into one having or not having a certain topological property.  In this game the topological property we consider is {\it connectedness}, but in a related game one could consider the property {\it $n$ or fewer connected components}.  Another class of topological combinatorial games played on the same game board requires players to resolve precrossings into crossings rather than smoothing.  Here the topological property being considered is {\it knottedness}. Games of this sort are discussed in~\cite{CMJ, Will}.

\begin{figure}[htbp] \begin{center}
\begin{tabular}{ccc}
\includegraphics[scale=.15]{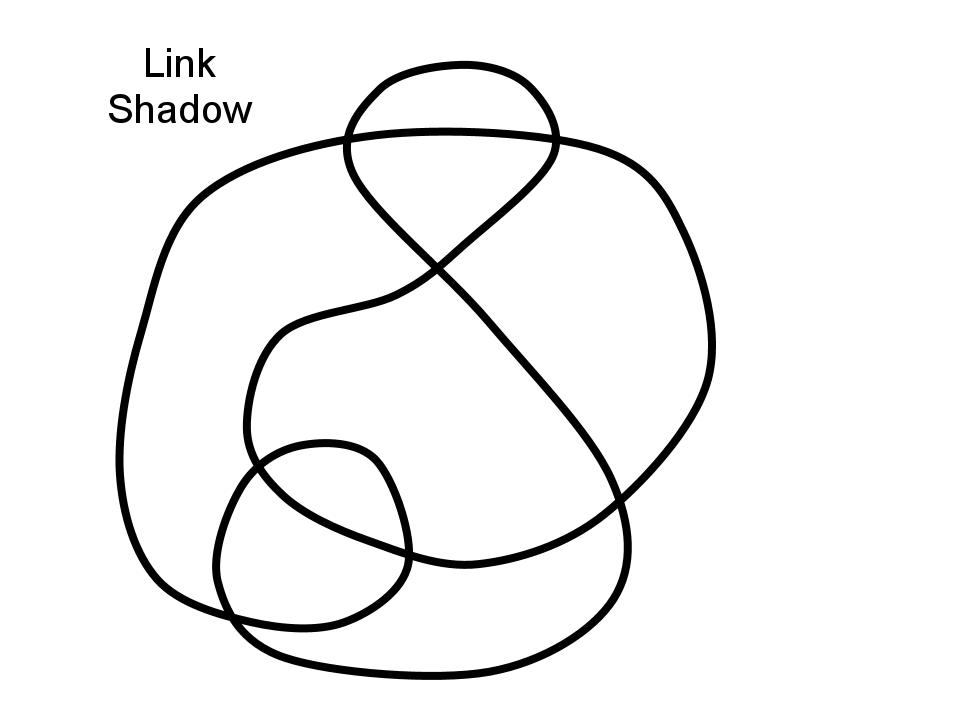}&\hspace{.8in} &\includegraphics[scale=.15]{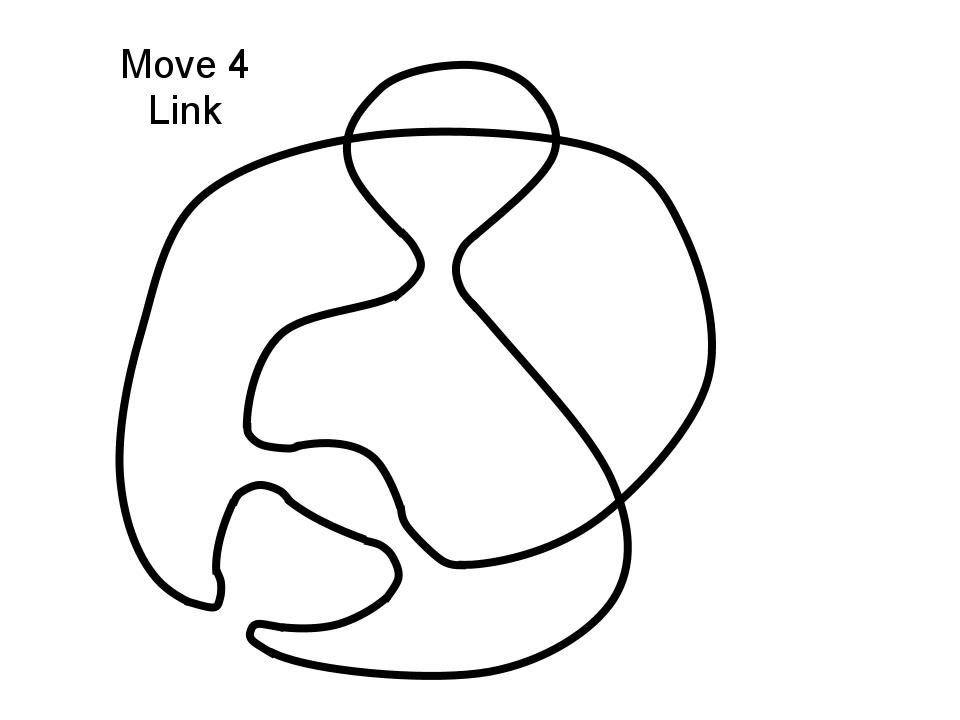}\\
\includegraphics[scale=.15]{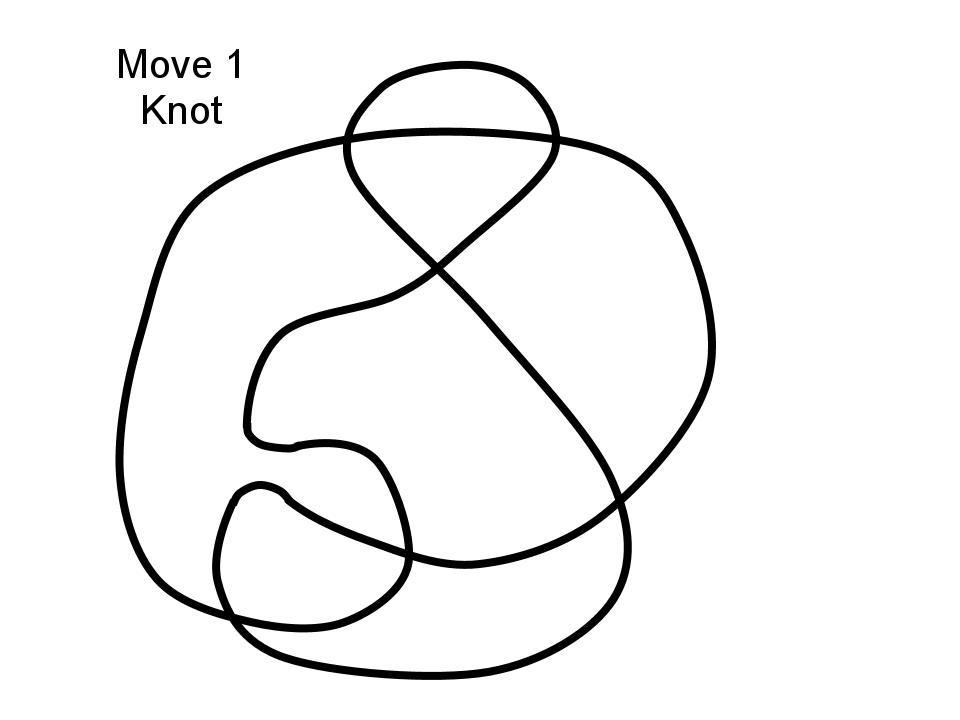}& &\includegraphics[scale=.15]{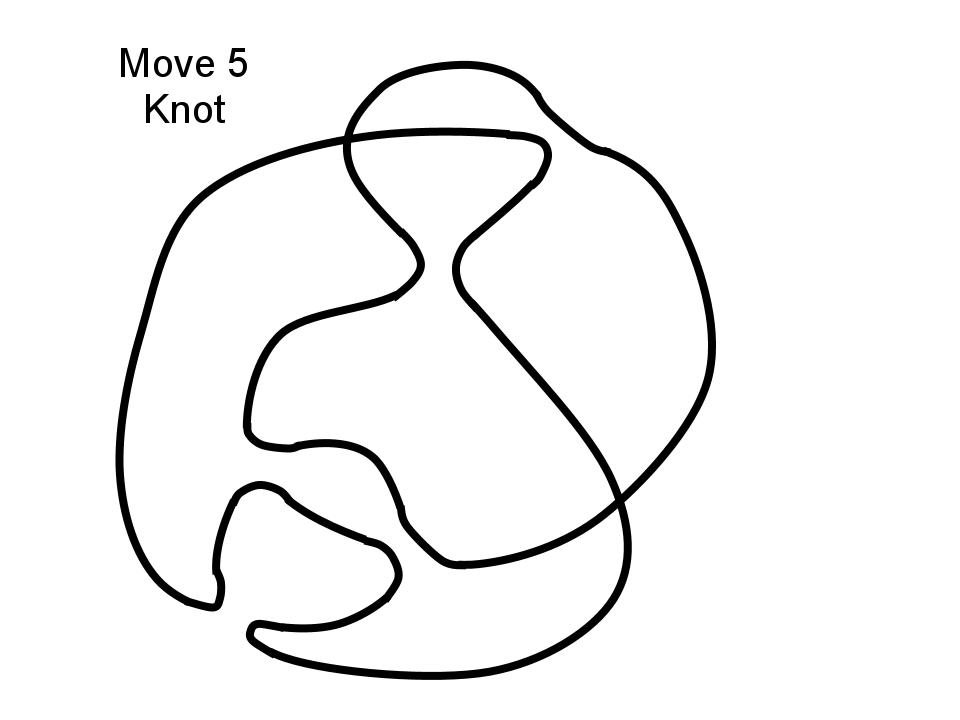}\\
\includegraphics[scale=.15]{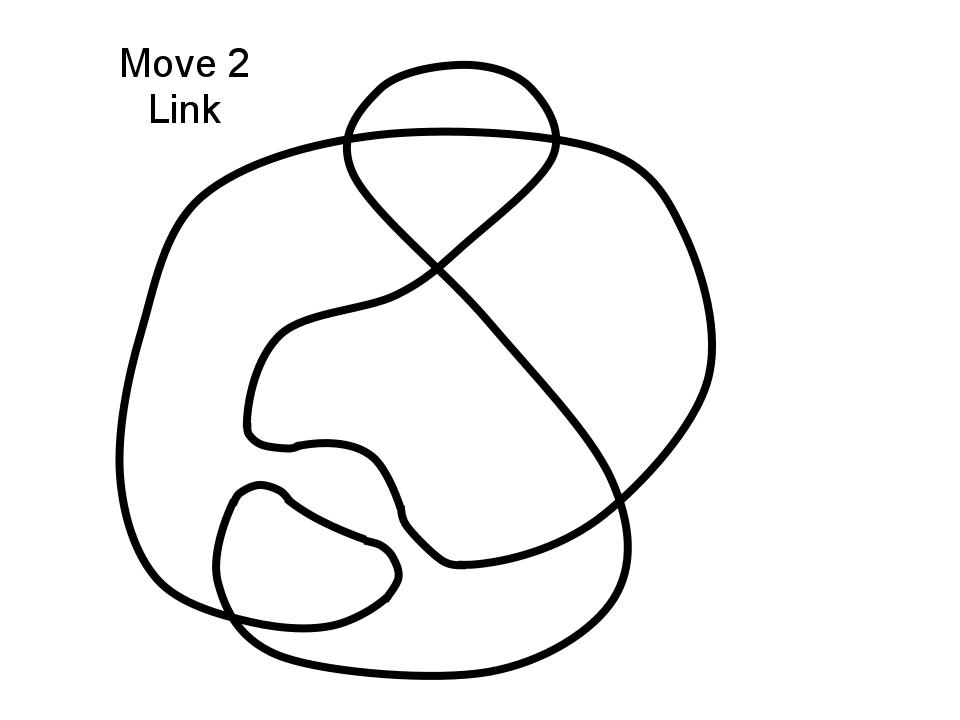}& &\includegraphics[scale=.15]{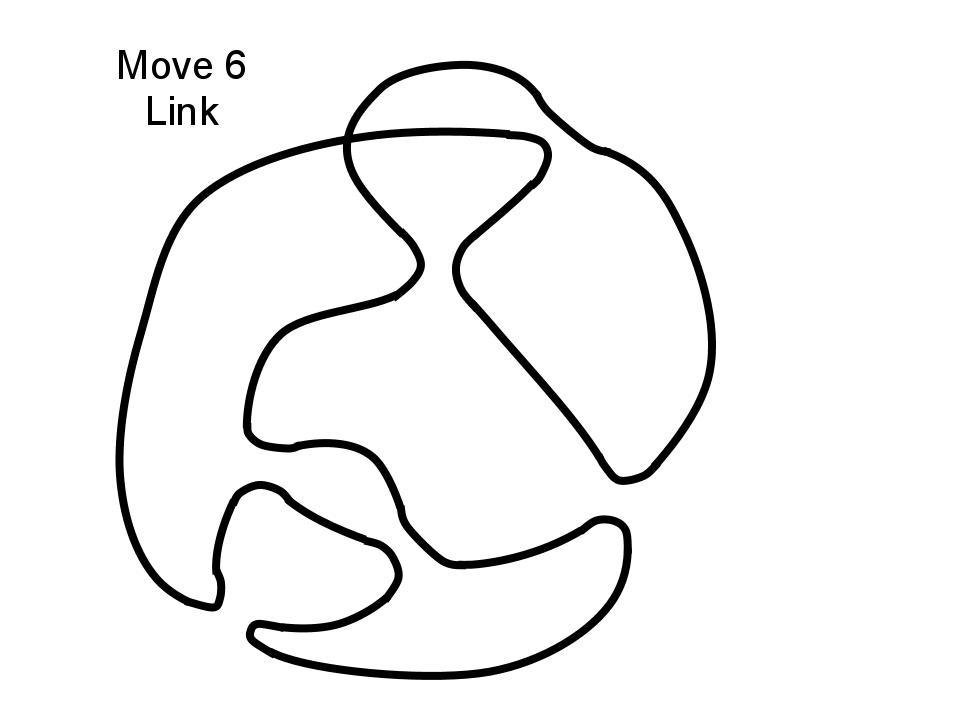}\\
\includegraphics[scale=.15]{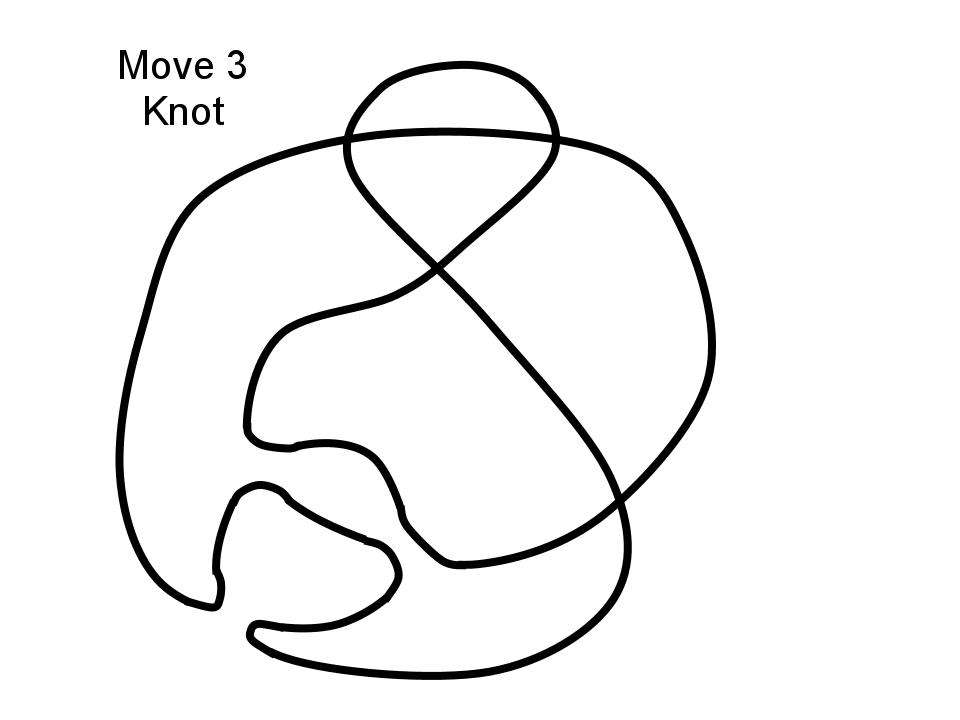}& &\includegraphics[scale=.15]{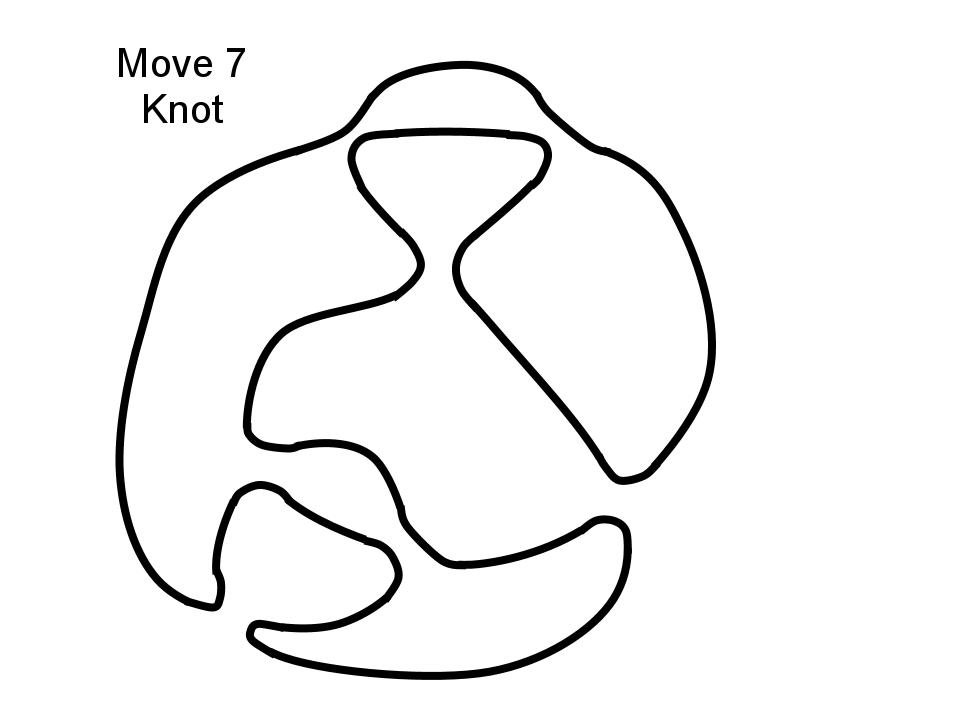}\\
\end{tabular}
\caption{{\bf Smoothing game example play on a 7 crossing link shadow. Knot plays first and Knot wins.}}
\label{gameplay}\end{center}
\end{figure}

For the purposes of this paper, our  goal is to classify link shadows by the outcomes of the Link Smoothing Game. In Section~\ref{smoothing}, we begin by making some immediate observations regarding link shadows on which $L$ has a winning strategy. We then illustrate how we can reinterpret our game as a game on planar graphs. In Section~\ref{classify}, we focus on a more complete classification of link shadows using their associated graphs. In Section~\ref{connect}, we examine the effect on the game of taking a connect sum of two shadows, and in Section~\ref{future}, we discuss what remains to be done in the classification of the Link Smoothing Game. 

\section{Graphs and the Link Smoothing Game.}\label{smoothing}

\subsection{Initial Observations.}
We begin our analysis of the Link Smoothing Game by making the following observation.

\begin{lemma}\label{L-adv} If L plays last in a game on link shadow $D$, then L has a winning strategy. Thus, 
\begin{enumerate}
\item If $D$ has an even number of crossings and L plays second, then L wins. 
\item If $D$ has an odd number of crossings and L plays first, then L wins.
\end{enumerate}
\end{lemma}

\begin{proof}
Any diagram with one remaining pre-crossing can be resolved so that the number of components either increases from one to two or remains at two or greater. \end{proof}

As a result of the previous proposition, it is clear that Link has a significant advantage in the game.  Indeed, Link can win before all crossings have been smoothed, while Knot can only win after the last crossing is smoothed.  In light of this, we investigate the following questions.  Are there link shadows on which Knot can guarantee a win? If so, how can these shadows be identified and what strategy can Knot use to win?  

Rephrasing these questions in the language of combinatorial game theory, we would like a classification of link shadows into outcome classes.  A link shadow $D$ is in one of four outcome classes, $\mathcal{N}$-position, provided that the {\it next} player to move (or first player) can guarantee a win; $\mathcal{P}$-position, provided that the previous player (or second player) can guarantee a win; $L$-position, provided that $L$ can guarantee a win regardless of moving first or second; and $K$-position, provided that $K$ can guarantee a win regardless of moving first or second.  Lemma~\ref{L-adv} implies there are no $K$-position link shadows.  Thus, as $K$ wins seem `rare,' we wish to identify all $\mathcal{N}$ and $\mathcal{P}$-position shadows. We shall assume $K$ moves last and describe the winning strategy for $K$.

We begin by looking at properties of shadows that are readily seen to be $L$-position shadows.
\begin{proposition}\label{shadow-Lwin}
Suppose L plays first on link shadow diagram $D$. Then L has a winning strategy if $D$ contains any of the following.
\begin{enumerate}
\item a nugatory crossing, i.e. a crossing that when appropriately smoothed disconnects the shadow,

\item a triplet of crossings as in Figure~\ref{triplecross}.
\end{enumerate} 
\end{proposition}

\begin{figure}[htbp] \begin{center}
\begin{center} \includegraphics[scale=.15]{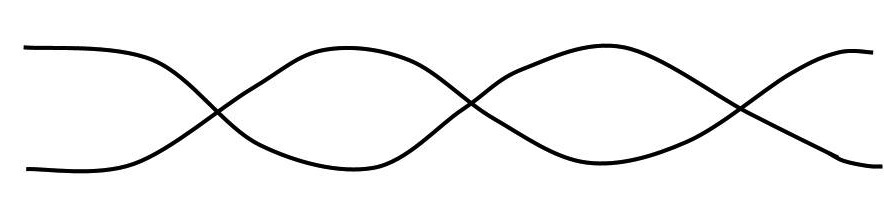}\end{center}
\caption{{\bf Triplet of crossings that guarantee a Link win.}}
\label{triplecross}\end{center}
\end{figure}

\begin{proof}
If a nugatory crossing is present, Link can win on her first move.  If Link vertically smooths the middle crossing of the triple, then she can win on her second move.
\end{proof}

\begin{figure}[htbp] \begin{center}
\begin{center} \includegraphics[scale=.15]{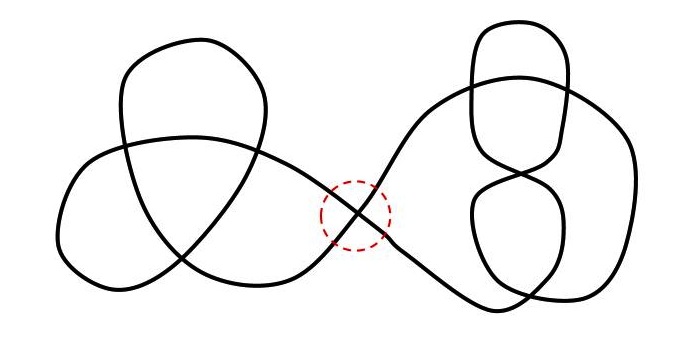}\end{center}
\caption{{\bf Example of a shadow with a nugatory crossing.}}
\end{center}
\end{figure}
\subsection{Links and Graphs.}

It is well known~\cite{Colin} that the shadow of a knot or link can be represented using a planar graph. This representation gives us tremendous insight into the game strategy.  

The graph, $G$, of the diagram $D$ is constructed by checkerboard coloring the regions of $D$, as in Figure~\ref{checker}. 

\begin{figure}[htbp] \begin{center}
\begin{center} \includegraphics[scale=.2]{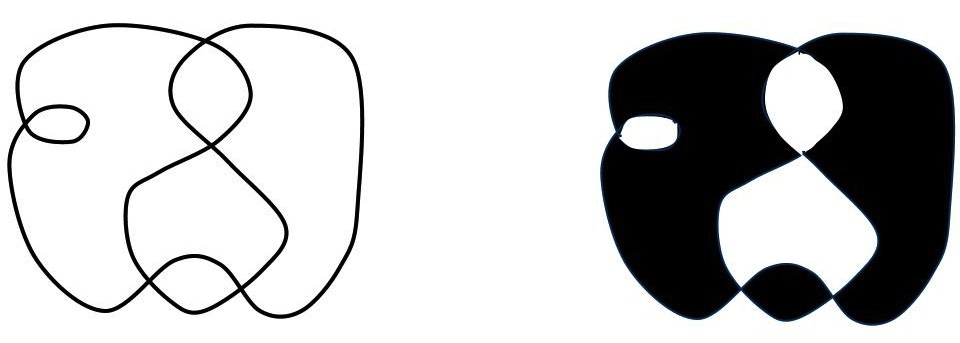}\end{center}
\caption{{\bf Example of a shadow and its checkerboard coloring.}}
\label{checker}\end{center}
\end{figure}

Vertices in the graph are in one-to-one correspondence with the black regions of the checkerboard coloring and there is an edge between vertices for each precrossing between the corresponding black regions.  The graph of a shadow may contain loops or multiple edges between a given vertex pair.  We call the graph of the shadow the black graph and its dual we call the white graph.

\begin{figure}[htbp]
\begin{center}
\includegraphics[scale=.2]{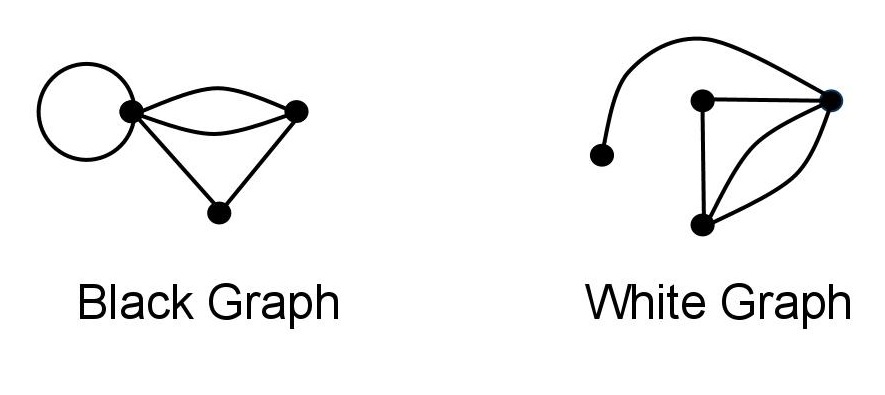}
\caption{{\bf The black and white graphs of the Figure~\ref{checker} shadow.}}
\label{Shadow2Graph}
\end{center}
\end{figure}

The Link Smoothing Game can be played on the black and white graphs rather than on the link shadow. We simply observe that a smoothing within a shadow corresponds either to an edge contraction or an edge deletion in the in the corresponding graph.  More precisely, an edge deletion in the black graph corresponds to a smoothing that eliminates an adjacency between black region(s) of the checkerboard coloring.  This smoothing in turn corresponds to joining white region(s) in the checkerboard coloring and thus contracting the corresponding edge in the white graph that is dual to the removed edge in the black graph.  Similarly, a smoothing that corresponds to an edge contraction in the black graph gives rise to the deletion of the dual edge in the white graph. We observe that the shadow is connected if and only if both the black and white graphs are connected.  

We note that while the graphs corresponding to link shadows often contain multi-edges, if the original link shadow is reduced (i.e. contains no loop that could be removed by a Reidemeister 1 move), then the corresponding black graph contains no loops.

In general, the strategy of the Link Smoothing Game is more readily seen in the graph representation of the shadow.  The following proposition is a reformulation of Proposition~\ref{shadow-Lwin} in terms of the graph of a link shadow.

\begin{proposition}\label{graph-Lwin}
Suppose $L$ plays first in a game on link shadow diagram $D$ corresponding to an embedding of connected planar graph $G$. Then $L$ has a winning strategy if $G$ contains any of the following.
\begin{enumerate}
\item a cut edge, i.e. an edge $e$ such that $G-e$ is disconnected,
\item a loop,
\item a pair of vertices joined by more than two edges,
\item adjacent degree two vertices that don't form a 2-cycle.
\end{enumerate} 
\end{proposition}
\begin{proof}
We first note that $L$ wins if a move is made such that either $G$ or its dual, $G^{\ast}$, becomes disconnected. Clearly, deleting a cut edge as in (1) disconnects the graph, yielding an $L$ win. In (2), we note that a loop in $G$ corresponds to a leaf in $G^{\ast}$, which is a particular type of cut edge. Thus, $L$ can disconnect the graph on her first move. If a pair of vertices is joined by more than two edges as in (3), then $L$ can contract one of the edges to produce two or more loops. At least one loop will remain in $L$'s second turn, so $L$ has a winning strategy as in (2). Finally, if $G$ has adjacent degree two vertices that don't form a 2-cycle, then $G^{\ast}$ is a game of type (3). Hence, $L$ wins.
\end{proof}

We will see that defensive strategies for Knot can be found for shadows whose graphs contain certain pairings of edges. We will use these strategies to identify many $\mathcal{P}$- and $\mathcal{N}$-position shadows and to exhibit the strategy Knot can follow to win.

\section{Classification of Games}\label{classify}

In this section, we prove that much of the link shadow classification into outcome classes is determined by the number of edges and vertices in the graph representation of the shadow, as indicated in Figure~\ref{GameGrid}.  First we show that link diagrams with isomorphic graphs are in the same outcome class.  This is a surprising fact as the link diagrams for isomorphic graphs can look quite different, as is seen in Figure~\ref{iso}.

\subsection{Isomorphic Graphs.}

\begin{lemma}\label{Lwin} Suppose a game is played on link shadow diagram $D$ corresponding to connected planar graph $G$. L's final move results in an L win iff L is allowed to play on a cut edge or a loop. 
\end{lemma}
\begin{proof}
By definition, a connected link diagram may only be disconnected by a smoothing at a crossing if a nugatory crossing is present. A nugatory crossing corresponds to a loop in one graph associated to the link diagram and a cut edge in the dual graph.
\end{proof}

\begin{figure}[htbp]
\begin{center}\includegraphics[scale=.22]{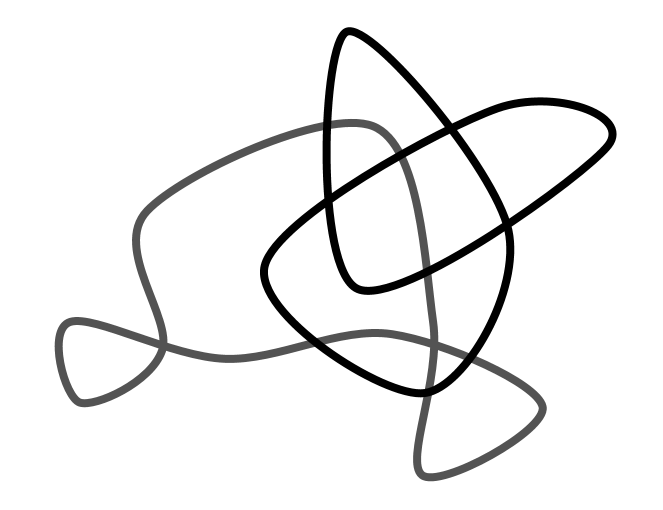}\hspace{.3in}\includegraphics[scale=.2]{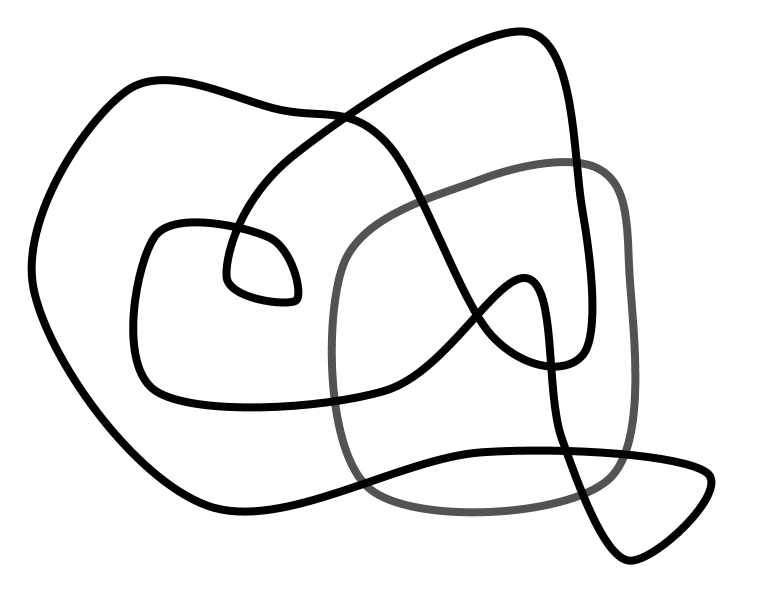}\end{center}
\caption{{\bf Link shadows with isomorphic black graphs.}}\label{iso}
\end{figure}

\begin{theorem}
Any two embeddings of a planar graph $G$ in the plane correspond to link shadow diagrams with equivalent game outcomes.
\end{theorem}

\begin{proof}
Suppose a link shadow $D$ has game outcome $o(D)$ under perfect play in the smoothing game. The strategy that produces $o(D)$ corresponds to a strategy (of edge deletion/contraction) on planar embeddings of the black and white graphs $G$ and $G^{\ast}$. By Lemma~\ref{Lwin}, $L$ wins if and only if at some point in the game, a cut edge or loop arises. Observe that if an isomorphic game is played on a different embedding of $G$ or $G^{\ast}$, a cut edge or loop will arise at the same point in this isomorphic game (or perhaps not at all), allowing $L$ a win (or $K$ a win, if such an edge is never produced). Thus, the particular embedding of the graph associated to the diagram $D$ is irrelevant to the winning strategy.
\end{proof}

\subsection{Classification.}

We have begun to see that $L$-position games are abundant. Figure~\ref{GameGrid} illustrates just how abundant they are. On the other hand, this figure also indicates that there are infinite families of graphs that are not $L$-position. In this section, we prove the major classification results that are summarized in Figure~\ref{GameGrid}.

\begin{figure}[htbp]
\begin{center}\includegraphics[scale=.4]{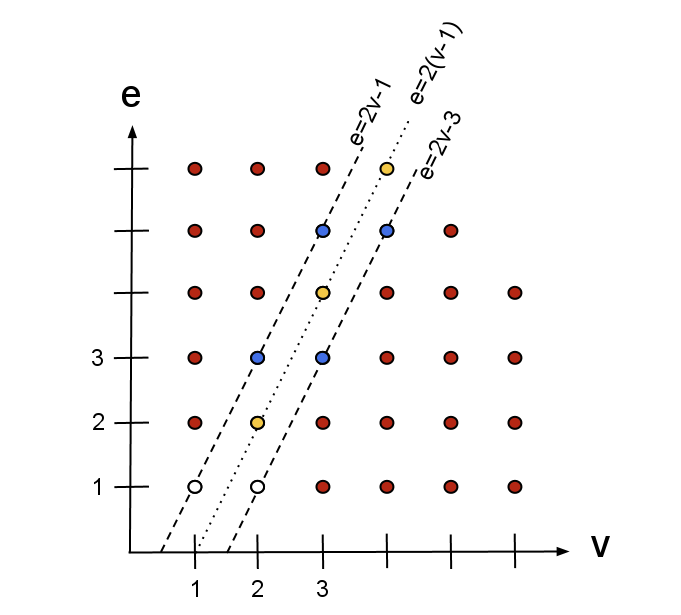}\end{center}
\caption{{\bf The game class of a link shadow with connected graph $G$ having $v$ vertices and $e$ edges is determined by the  color of the lattice point $(v, e)$.  Red indicates an $L$-position graph.  Blue, along the lines $e=2v-1$ and $e=2v-3$,  indicates the existence of both $\mathcal{N}$-position and $L$-position graphs but no $\mathcal{P}$-position graphs.  Yellow, along the line $e=2(v-1)$, indicates the existence of both $\mathcal {P}$-position and $L$-position graphs  but no $\mathcal{N}$-position graphs.  The pairs $(2, 1)$ and $(1, 1)$ are $\mathcal{N}$-position.}}\label{GameGrid}
\end{figure}

\begin{theorem}\label{count} Let $G$ be a connected, planar graph associated to a link shadow $D$.  Let $v$ and $e$ denote the number of vertices and edges, respectively, in $G$.  If $e$ is even and $2(v-1)\neq e$, then $D$ is an $L$-position diagram.
\end{theorem}
\begin{proof}  Since $e$ is even, we may assume $L$ moves first on $D$, else $L$ has the last move and wins. 
Suppose it is the case that $e> 2(v-1)$.  We prove $L$ has a winning strategy on $D$ by induction on $v$.  The assumptions above applied to the base case $v=1$ gives rise to a graph $G$ that is a bouquet of $e \geq 2$ loops.  In $L$'s first move, she can contract a loop, thus disconnecting the dual graph and winning.  

Let $G$ be a connected planar graph with $v>1$ vertices and an even number of edges, $e$, such that $e>2(v-1)$.   Suppose $L$ has a winning strategy on all such graphs with fewer than $v$ vertices.  In $L$'s first move she can contract any edge of $G$, thus resulting in a connected planar graph with $v-1$ vertices and $e-1$ edges.  In $K$'s next move she can either contract or delete an edge resulting in $e-2$ edges and either $v-2$ or $v-1$ vertices respectively.  Presumably, $K$'s move will not disconnect the diagram, thus it suffices to prove $e-2 > 2(v-2-1)$ and $e-2> 2(v-1-1)$. The result, then, follows by induction since these two inequalities follow readily from the assumption that $e>2(v-1)$.

Next suppose $e<2(v-1)$.  We will prove that the dual of $G$, denoted $G^{\ast}$ with $v^{\ast}$ vertices and $e^{\ast}$ edges, satisfies $e^{\ast}>2(v^{\ast}-1)$.  Proving this inequality holds will imply $L$ has a winning strategy on the shadow associated to $G^{\ast}$ by the argument given above, and thus, by duality, $L$ has a winning strategy on $D$. (Recall that if one of $G$ or $G^{\ast}$ can be disconnected, $L$ wins.) The planar graph $G$, embedded on the sphere gives rise to a polygonal decomposition satisfying $v-e+f=2$, where $f$ is the number of polygon faces.  By definition of the dual graph $e^{\ast}=e$, $v^{\ast}=f=2+e-v$.  Thus $e<2(v-1)$ implies $e^{\ast}<2(2+e^{\ast}-v^{\ast}-1)$, which implies $e^{\ast}>2(v^{\ast}-1)$.
\end{proof}

From this theorem, we derive an immediate corollary for graphs with an odd number of edges.

\begin{corollary}\label{ve-cor} Let $G$ be a connected, planar graph associated to a link shadow $D$.  Let $v$ and $e$ denote the number of vertices and edges, respectively, in $G$.  If $e$ is odd and $e>2v-1$or $e<2(v-1)-1$, then $D$ is an $L$-position shadow.
\end{corollary}
\begin{proof} First, we note that if $L$ moves first, $L$ wins, so we assume that $K$ moves first. It is easy to verify that no move performed by $K$ will result in a diagram such that $e=2(v-1)$. Thus, after $K$'s move, we must be in the situation of Theorem~\ref{count}.
\end{proof}

We note that the condition $e=2(v-1)$ on the graph $G$ of a diagram $D$ is necessary for $K$ to have a defensive winning strategy, but it doesn't guarantee a win. This is illustrated in Figure~\ref{noKwin}. Each of these graphs satisfy the condition $e=2(v-1)$, however they are all $L$-position graphs.

\begin{figure}[htbp]
\begin{center}\includegraphics[scale=.3]{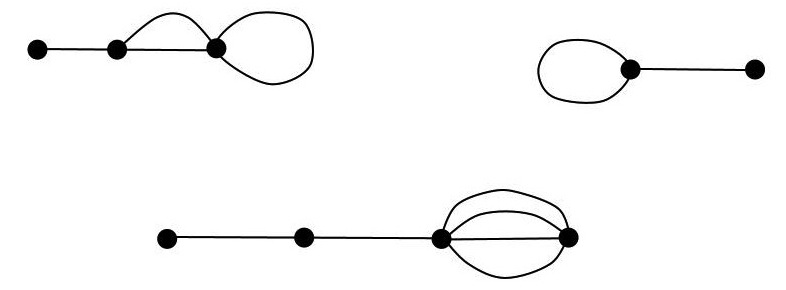}\end{center}
\caption{{\bf  $L$-position graphs satisfying $2(v-1)=e$.}}\label{noKwin}
\end{figure}

On the other hand, there are many examples where $K$ has a defensive winning strategy. Here, we give a characteristic that gives rise to many $\mathcal{P}$-position graphs. The sufficient condition we provide for a graph to be $\mathcal{P}$-position makes use of the notion of a spanning tree. We refer the reader unfamiliar with this concept to an introductory graph theory text such as~\cite{pearls}.

\begin{theorem}\label{spanning} Let $G$ be a connected, planar graph associated to a link shadow $D$. Suppose that $G$ is the union of two subgraphs, $T_1$ and $T_2$, where $T_1$ and $T_2$ are edge-disjoint spanning trees for $G$. Then $G$ is a $\mathcal{P}$-position game.
\end{theorem}
\begin{proof} We note that each spanning tree in $G$ has $v$ vertices and $v-1$ edges, so $G=T_1\cup T_2$ with $e(T_1)\cap e(T_2)=\emptyset$ implies that $e=2(v-1)$. Furthermore, if a graph is composed of two edge-disjoint spanning trees, so is its dual. Thus, we can limit our attention to the black graph of a link shadow. Since our graph has an even number of edges, we assume that $L$ moves first. Our proof proceeds by induction. 

The basis case is clear, so let us suppose that $K$ has a winning strategy moving second on any graph with $2n$ edges that is the union of two edge-disjoint spanning trees. Let $G$ be a connected planar graph with $2(n+1)$ edges that is the union of two edge-disjoint spanning trees $T_1$ and $T_2$. Suppose, without loss of generality, that $L$ contracts edge $e_1$ in $T_1$. $T_1/e_1$ remains a connected spanning tree, but $T_2$ must now contain a cycle. Deleting any edge in $T_2$'s cycle returns $T_2$ to being a connected spanning tree. The result is a graph with $2n$ edges of the appropriate form to guarantee a $K$ win, by the inductive hypothesis. 

Suppose instead that $L$ deletes edge $e_1$ from $T_1$ on his turn. Then $T_1\backslash e_1$ is disconnected (and may simply be the disjoint union of a tree with a single vertex), while $T_2$ remains connected. Since $T_1$ is a spanning tree of $G$ the distance between the two components of $T_1\backslash e_1$ is 1. Thus as $T_2$ spans $G$, there must be an edge in $T_2$ that, when contracted, reunites the two components of $T_1$, returning it to being a spanning tree. This contraction in $T_2$ preserves its spanning tree property in the resulting graph, so the outcome of the deletion-contraction is a graph with $2n$ edges composed of two edge-disjoint spanning trees. Thus, $K$ has a winning strategy by induction.
\end{proof}

The $\mathcal{N}$-position games referred to in Figure~\ref{GameGrid} are discussed at the end of the following section of examples of $\mathcal{P}$-position graphs.

\subsection{Examples.}
There are many interesting examples of planar graphs containing two edge-disjoint spanning trees. One of the simplest families with this property is the family of wheel graphs, shown in Figure~\ref{wheel}.

\begin{figure}[htbp]
\begin{center}\includegraphics[scale=.3]{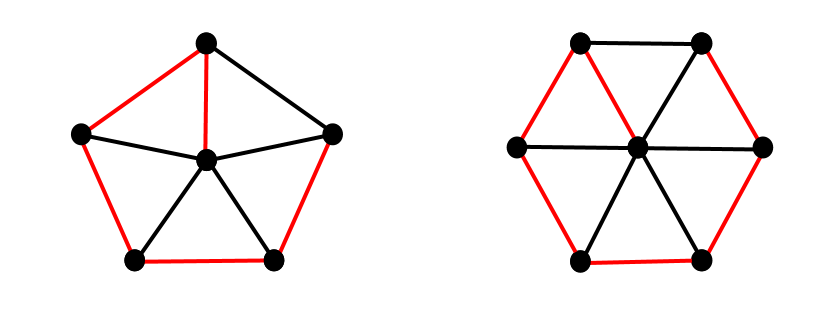}\end{center}
\caption{{\bf Wheel graphs and edge-disjoint spanning trees.}}\label{wheel}
\end{figure}

Another interesting family of $\mathcal{P}$-position graphs containing two edge-disjoint spanning trees are called \textbf{2-edge blowout graphs}. To introduce these examples, we first define an $n$-edge blowout.

\begin{definition} Let $G$ be a finite planar graph with a fixed planar embedding $D$ and  let $v$ be a vertex of $G$.  An \textbf{$n$-edge blowout of $G$ at $v$} is a planar graph obtained from $D$ by selecting a partition of the edges emanating from $v$ into two subsets of edges (as pictured in Figure~\ref{blow}), `splitting' the  vertex $v$ into two vertices $v'$ and $v''$ so that each of the two edge subsets are connected to $v'$ or $v''$, and then adding $n$ new edges to the diagram between $v'$ and $v''$ in any way that results in a planar graph.
\end{definition} 

\begin{figure}[htbp]
\begin{center}
\includegraphics[scale=.2]{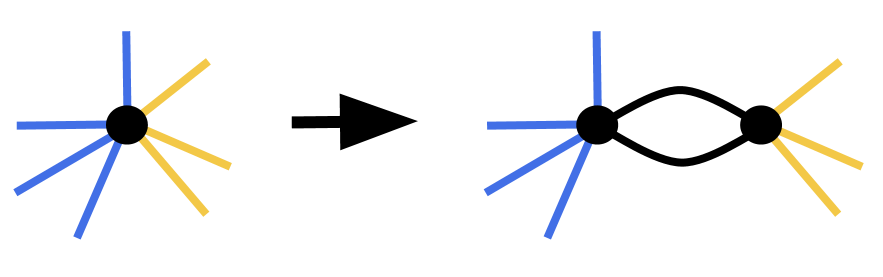}
\end{center}
\caption{{\bf The local picture of a 2-edge blowout at $v$. On the left, the partition of edges is denoted by the edge colors yellow and blue. On the right is the blowout corresponding to this partition. }}\label{blow}
\end{figure}

Perhaps surprisingly, the edge blowout operation can be used repeatedly to create any connected planar loopless graph.

\begin{proposition} Given a finite connected planar loopless graph $G$ with $k+1$ vertices, there exists a sequence of integers $n_1, n_2, \dots, n_k$ such that $G$ can be built from a single vertex via a sequence of edge blowouts of type $n_1, n_2, \dots, n_k$.  
\end{proposition}
\begin{proof} We use induction on the number of vertices in $G$.  Any connected planar loopless graph, $G$, with two vertices is simply an $n$-edge blowout on a single vertex, where $n$ is the number of edges in $G$.
Let $G$ be a finite connected planar loopless graph $G$ with $k+1$ vertices.  Pick any two adjacent vertices, call them $v_1$ and $v_2$.  Then $v_1$ and $v_2$ will have some number of edges (greater than or equal to 1) between them.  Call this number $n_k$.  Collapse $G$ by identifying $v_1$ and $v_2$ and deleting all $n_k$ edges between $v_1$ and $v_2$.    This collapse, being the dual operation to an $n_k$-edge blowout, results in a new connected planar loopless graph $G'$ with $k$ vertices.  So by induction there is a sequence of $n_1, n_2, \dots, n_{k-1}$-edge blowouts that can be used to build $G'$ from a single vertex.  Thus $n_1, n_2, \dots, n_{k-1}, n_k$-edge blowouts can be applied to build $G$.
\end{proof}

We note that the sequence of blowouts used to build a graph $G$ is not unique, as seen in Figure~\ref{cyl}.

\begin{figure}[htbp]
\begin{center}
\includegraphics[scale=.3]{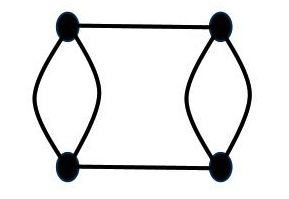}
\end{center}
\caption{{\bf A graph built from a single vertex via a $4, 1, 1$ blowout sequence or a $2, 2, 2$ blowout sequence.}}\label{cyl}
\end{figure}

Now, we return to our quest for examples of graphs containing two edge-disjoint spanning trees.

\begin{theorem}\label{2blow}
Suppose shadow diagram $D$ corresponds to a graph $G$. Let $G$ be a graph that can be constructed from a single vertex by a sequence of $2$-edge blowouts. Then $G$ contains two edge-disjoint spanning trees. Thus, $D$ is a $\mathcal{P}$-position game.
\end{theorem}

\begin{proof} It is easy to verify that if one chooses a single edge coming from each 2-edge blowout in the blowout sequence, the resulting collection of edges defines a spanning tree. Similarly, the complement is a spanning tree.   
\end{proof}

The following lemma, regarding the $\mathcal{N}$-position games along the lines $e=2v-1$ and $e=2v-3$ in Figure~\ref{GameGrid}, is easier to state now with the language of blowouts.

\begin{corollary} Suppose  shadow diagram $D$ corresponds to a graph $G$.  Suppose $G$ is a $\mathcal{P}$-position game.  Then any planar graph obtained from $G$ by performing a 1-blowout or adding a new edge between two existing vertices will result in a graph that is an $\mathcal{N}$-position game.
\end{corollary}

\begin{proof}  Let the planar graph obtained from $G$ by performing a 1-blowout or adding a new edge between two existing vertices be called $G_*$.  As $G_*$ must have an odd number of edges we must only consider the case when $K$ moves first and show she has a winning strategy.  Clearly, she can contract the blowout edge or delete the added edge within $G_*$, thus leaving the $\mathcal{P}$-position graph $G$ for $L$'s move.  
\end{proof}


\section{Connect Sums.}\label{connect}

An elementary operation that can be performed on two knot diagrams or on their shadows is to take a \emph{connect sum}. A connect sum $D_1\# D_2$ is obtained from the two diagrams $D_1$ and $D_2$ by removing a small arc segment from an exterior arc of each diagram and connecting the two diagrams at the newly-formed endpoints by a set of parallel curves, as in Figure~\ref{connectsum}.

\begin{figure}[htbp]
\begin{center}\includegraphics[scale=.17]{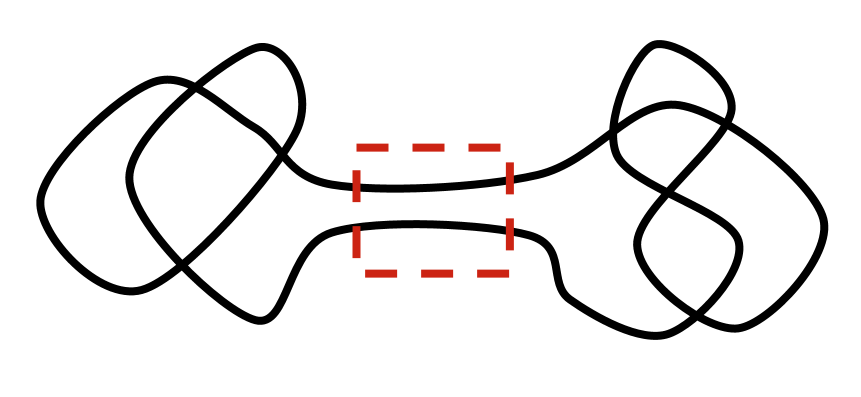}\end{center}

\begin{center}\includegraphics[scale=.17]{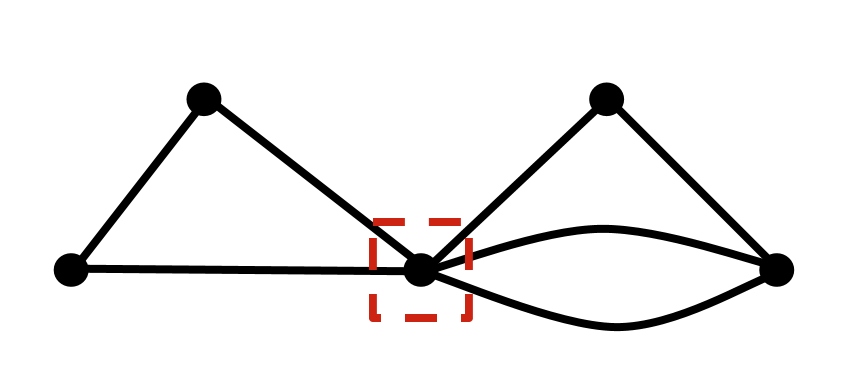}\end{center}
\caption{ {\bf A connect sum of trefoil and figure eight shadows and the associated black graph.}}\label{connectsum}
\end{figure}

Notice that if we look at the black or white graphs instead of the link shadow diagrams, we see that the connect sum operation identifies an exterior vertex of one graph with an exterior vertex of the other graph.  If $G_1$ and $G_2$ are the respective black graphs for $D_1$ and $D_2$, we write $G_1 \vee G_2$ for the wedge of graphs associated to the connect sum.
We immediately have the following result.

\begin{proposition} Suppose that $D_1$ and $D_2$ are link shadow diagrams.
\begin{enumerate}
\item If $D_1$ and $D_2$ are both $\mathcal{P}$-position, then $D_1\# D_2$ is $\mathcal{P}$-position.
\item If $D_1$ is  $\mathcal{P}$-position and $D_2$ is $\mathcal{N}$-position, then $D_1\# D_2$ is $\mathcal{P}$-position.
\item If $D_1$ and $D_2$ are both $\mathcal{N}$-position, then $D_1\# D_2$ is $L$-position.
\item Any connect sum with an $L$-position summand is $L$-position.
\end{enumerate}
\end{proposition}

\begin{proof} The proofs of cases 1 and 4 are straightforward. Below, we reduce the case of the second assertion to that of the first and reduce the third assertion to the fourth. 

Suppose that $D_1$ is  $\mathcal{P}$-position and $D_2$ is $\mathcal{N}$-position as in case 2. This implies that $G_1$ has an even number of edges and $G_2$ has an odd number of edges. It suffices to prove $K$ has a winning strategy on $G_1\vee G_2$ moving first. Knot's first move in $G_1\vee G_2$ will correspond to the move in $G_2$ that produces a $\mathcal{P}$-position shadow from $D_2$. This reduces the game to case 1.

Next suppose that  $D_1$ and $D_2$ are $\mathcal{N}$-position as in 3.  This implies that the vertex and edge counts for $G_1$ and $G_2$ satisfy either $e=2v-1$ or $e=2v-3$.  Theorem~\ref{count} and Corollary~\ref{ve-cor} imply the graph $G_1 \vee G_2$ must be $L$-position if at least one of the two graphs satisfies $e=2v-3$.  In the case when both $G_1$ and $G_2$ satisfy  $e=2v-1$, we need only consider the case where Link moves first and the total edge count is even.  Suppose Link's first move is to contract any edge in say $G_1$.  We denote the the resulting graph as $H \vee G_2$, as it is still a one point union of $G_2$ with some graph $H$ that comes from applying the contraction to $G_1$.  Now, the edge and vertex count of $H$ imply it is the graph of an $L$-position shadow.  Thus we have reduced the third case to the fourth.
\end{proof}

 \section{Future Work.}\label{future}
We have made significant progress in determining the outcome classes of link shadows for the Link Smoothing Game. To have a complete classification of link shadows, more work must be done with graphs that have vertex--edge pairs lying on the three lines pictured in Figure~\ref{GameGrid}. Specifically, we would like to prove or disprove that  $\mathcal{P}$-position graphs are exactly those that are the union of two edge-disjoint spanning trees.

In addition to completing the classification of this game, we would like to consider a generalization of the Link Smoothing Game which we refer to as the $n$-Connected Link Smoothing Game. Here, the two players are allowed to smooth precrossings on a disjoint union of link shadows, but one player wants the game to result in $n$ or fewer connected components while the other player wants more than $n$ connected components.  The Link Smoothing Game is then the $n$-Connected Link Smoothing Game for $n=1$.  

The $n$-Connected Link Smoothing Game could be of additional interest because it gives a filtration of the set of all link shadows and disjoint unions of link shadows.   Let $P_n$ denote the set of all $\mathcal{P}-$position shadows for the $n$-Connected Link Smoothing Game and $P_0 = \emptyset$.  Clearly, if a link shadow is in $P_n$, then it is also in $P_{n+k}$ for all $k \geq 1$.   Also, every link shadow with $n$ crossings is in $P_{n+1}$ because it can be disconnected into at most $n+1$ pieces.  Thus it may be interesting to study  the sets $P_{n} \setminus P_{n-1}$ for all $n\geq 1$, in order to determine for a given link shadow the smallest $n$ for which $P_n$ contains that shadow. 
 
 \section{Acknowledgements.}
The authors would like to thank Louis Kauffman for his suggestions regarding the idea for this game and Jonathan Zung for his insightful idea for Theorem~\ref{spanning}.

\bibliographystyle{amsalpha}

 \end{document}